%% file: arxivcollectivemotion.tex
\documentclass[11pt]{article}
\usepackage{authblk}

\usepackage[utf8]{inputenc}
\usepackage[T1]{fontenc}
\usepackage[utf8]{inputenc}
\usepackage{graphicx}
\usepackage{amsmath}
\usepackage{mathtools}
\usepackage{marginnote}

\usepackage[normalem]{ulem}

\usepackage{amsmath,amssymb,amsthm}
\usepackage{paralist}
\usepackage{verbatim}
\usepackage[right]{eurosym} 
\usepackage{url}
\usepackage{enumerate}
\usepackage{tikz}   

\usepackage[top=2.8cm,bottom=2.8cm,left=2.5cm,right=2.5cm]{geometry}
\usepackage[colorlinks=true, allcolors=blue]{hyperref}

\newtheorem{theorem}{Theorem}

\newtheorem{proposition}[theorem]{Proposition}

\newtheorem{corollary}[theorem]{Corollary}

\newtheorem{remark}[theorem]{Remark}

\renewcommand{\Re}{\mathrm{Re}}
\renewcommand{\Im}{\mathrm{Im}}

\newcommand*{\N}{\mathbb{N}}
\newcommand*{\R}{\mathbb{R}}

\newcommand*{\C}{\mathbb{C}}

\newcommand{\iu}{\mathrm{i}} 

\newcommand*{\B}{B}

\newcommand*\xbar[1]{%
	\,\hbox{%
		\vbox{%
			\hrule height 0.1pt
			\kern0.4ex%
			\hbox{%
				\kern-0.1em%
				\ensuremath{#1}%
				\kern-0.1em%
			}%
		}%
	}\,%
} 

\title{Stabilisation of stochastic single-file dynamics using port-Hamiltonian systems}

\author{Julia Ackermann\footnote{Research Group Applied and Computational Mathematics, \href{mailto:jackermann@uni-wuppertal.de}{jackermann@uni-wuppertal.de}} ,
Matthias Ehrhardt\footnote{Research Group Applied and Computational Mathematics,  \href{mailto:ehrhardt@uni-wuppertal.de}{ehrhardt@uni-wuppertal.de}}, 
Thomas Kruse\footnote{Research Group Applied and Computational Mathematics,  \href{mailto:tkruse@uni-wuppertal.de}{tkruse@uni-wuppertal.de}}, 
Antoine Tordeux\footnote{Research Group Traffic Safety and Reliability,  \href{mailto:tordeux@uni-wuppertal.de}{tordeux@uni-wuppertal.de}} 
}
\affil{IMACM, School of Mathematics and Natural Sciences, \\ University of Wuppertal, Germany}
\date{January 2024}

\begin{document}
	
	\maketitle
	
	\begin{tikzpicture}[remember picture,overlay]
		\node[anchor=north east,inner sep=20pt] at (current page.north east)
		{\includegraphics[scale=0.2]{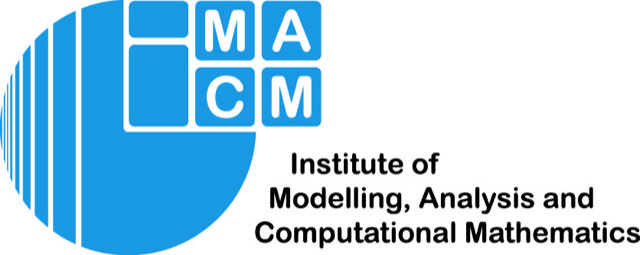}};
	\end{tikzpicture}
	
	\begin{abstract}
		This study revisits a recently proposed symmetric port-Hamiltonian single-file model in one dimension. 
		The uniform streaming solutions are stable in the deterministic model. However, the introduction of white noise into the dynamics causes the model to exhibit divergence. 
		In response, we introduce a control term in a port-Hamiltonian framework.
		Our results show that this control term effectively stabilises the dynamics even in the presence of noise, providing valuable insights for the control of road traffic flows.
	\end{abstract}
	
	\begin{minipage}{0.9\linewidth}
		\footnotesize
		\textbf{AMS classification:} 76A30, 82C22, 60H10, 37H30.
		\medskip
		
		\noindent
		\textbf{Keywords:} Port-Hamiltonian systems, 
		multi-agent systems,
		asymptotic stabilisation,
		modelling and control of road traffic flows 
	\end{minipage}

\section{Introduction}

The stability of single-file dynamics is of great interest in traffic engineering. 
Pioneering work by Reuschel, Pipes, and others 
\cite{reuschel1950fahrzeugbewegungen,pipes_OperationalAnalysisTraffic_1953,kometani1958stability,chandler1958traffic} have shown that single file traffic flow can present stability problems. 
Although the problem dates back to the 1950s and early 1960s and has been actively researched since then, see the review of
Wilson \& Ward~\cite{wilson_CarfollowingModelsFifty_2011}, it is still relevant today.
Even current adaptive cruise control (ACC) systems exhibit unstable dynamics in experiments, cf.\ the results in
Gunter et al.~\cite{gunter2020commercially}, Makridis et al.~\cite{makridis2021openacc} and Ciuffo et al.~\cite{CIUFFO2021}.  

In fact, many factors can perturb the stability of single-file uniform streaming. 
It is a well-known fact that lag and delay in the dynamics can lead to \textit{linear instability} of the uniform equilibrium solution, see 
Gasser et al.~\cite{gasser2004bifurcation},
Orosz et al.~\cite{orosz2010traffic},
Tordeux et al.~\cite{tordeux_LinearStabilityAnalysis_2012,tordeux2018traffic}. 
More recent approaches show that additive stochastic noise can also perturb single-file dynamics, even for stable deterministic systems, cf.\  
Treiber \& Helbing~\cite{treiber2009hamilton},
Tordeux \& Schadschneider~\cite{tordeux2016white},
Treiber \& Kesting~\cite{treiber2017intelligent}, 
Friesen et al.~\cite{friesen2021spontaneous}, 
Ehrhardt et al.~\cite{ehrhardt2023collective}.

Port-Hamiltonian systems (pHS) have recently emerged as a modelling framework for nonlinear physical systems, cf.\ van der Schaft \& Jeltsema~\cite{van2014port}.  
Unlike conservative Hamiltonian systems, pHS integrate control and external forcing into the dynamics through input and output ports. 
PHS can represent systems from diverse physical domains such as thermodynamics, electromechanics, electromagnetics, fluid mechanics and hydrodynamics, cf.\ Rashad et al.~\cite{rashad2020twenty}. 
The functional structure of pHS, which extends the Hamiltonian framework with dissipation, input and output ports, serves as a meaningful representation for a wide range of physical systems. 

Recent research shows that pHS provide valuable modelling techniques for multi-agent systems, including multi-input multi-output (MIMO) systems
(Sharf \& Zelazo~\cite{sharf2019analysis}), 
swarm behaviour (Matei et al.~\cite{matei2019inferring}), 
interacting particle systems (Jacob \& Totzeck~\cite{jacob2023port}), 
pedestrian dynamics
(Tordeux \& Totzeck~\cite{tordeux2022multi}), 
or autonomous vehicles and adaptive cruise control systems
(Knorn et al.~\cite{knorn2014scalability}, Dai \& Koutsoukos \cite{dai2020safety}). 
These micro-level agent-based models are constructed using finite-dimensional pHS. 
In addition, macroscopic traffic flow models 
(Bansal et al.~\cite{bansal2021port}) 
and more general fluid dynamics models
(Rashad et al.~\cite{rashad2021portb}) 
use infinite-dimensional pHS.

In this paper, we first recall a symmetric port-Hamiltonian single-file model recently introduced to describe collection motions in one dimension, see Ehrhardt et al.~\cite{ehrhardt2023collective}. 
Interestingly, the uniform solutions are stable for the deterministic model. 
However, the introduction of white noise in the dynamics causes the model to diverge. 
We show that the introduction of a control term - the input port in the port-Hamiltonian formulation - allows the dynamics to be stabilised even in the presence of noise.

The paper is structured as follows. 
In the next section, we introduce the port-Hamiltonian single-file model, elucidating its intricacies, and additionally, we expound upon its extension with the inclusion of the input control term. The analysis of the long-term behavior of the models, coupled with an exploration of their behavior as the number of agents $N$ approaches infinity, is presented in Section~\ref{long}. 
Some simulation results with a line of 10 agents are shown in Section~\ref{sim}. 
Section~\ref{ccl} is dedicated to discussions that not only encapsulate the key findings but also propose potential avenues for further extensions of the model.

\section{Port-Hamiltonian single-file models}\label{def}
Let us first start with the notations and the definitions of the models we are considering.

\paragraph{Notations.}
We consider $N\in \{3,4,\dots\}$ agents on a segment of length $L$ with periodic boundaries.
We denote 
\begin{equation*}
q(t)=\bigl(q_n(t)\bigr)_{n=1}^N\in\R^N,\qquad t\in [0,\infty),
\end{equation*}
the positions of the agents and 
\begin{equation*}
p(t)=\bigl(p_n(t)\bigr)_{n=1}^N\in\R^N,\qquad t\in [0,\infty),
\end{equation*}
are the velocities of the agents at time $t$.
We assume that the positions $q(t)$ and the velocities $p(t)$ of the $N$ agents are known at time $t=0$, 
\begin{equation*}
p(0)=p|_{t=0}= \mathfrak{p}\in\R^N,\qquad 
q(0)=q|_{t=0}= \mathfrak{q}\in\R^N,
\end{equation*}
and that the positions of the agents are initially ordered by their indices, i.e.,
\begin{equation}\label{eq:ini_order}
0\le q_1(0)\le q_2(0)\le\ldots\le q_N(0)\le L.
\end{equation}
Furthermore, in the following we consistently use the index $n+1$ to 
represent the nearest neighbour on the right and $n-1$ for the nearest neighbour on the left.
The right neighbour of the $N$-th agent is the first agent, i.e., $n+1=1$ when $n=N$. 
Conversely, the left neighbour of the first agent is the $N$-th agent, i.e., $n-1=N$ if $n=1$.

The distances to the right neighbours (see Fig.~\ref{fig:scheme})
\begin{equation*}
Q(t)=(Q_n(t))_{n=1}^N\in\R^N,\qquad t\in [0,\infty),
\end{equation*}
are given by
\begin{equation}\label{eq:def_dist}
\begin{cases}
Q_n(t)&=q_{n+1}(t)-q_n(t), \quad n\in\{1,\dots,N-1\},\\
Q_N(t)&=L+q_1(t)-q_N(t).
\end{cases}
\end{equation}
The distance to the left is $Q_N$ for the first agent and $Q_{n-1}$ for the $n$-th agent, $n\in\{2,\dots,N\}$. 
The index order of the agents at time zero makes the initial distance positive 
\begin{equation*}
Q(0)=Q|_{t=0}:=\mathfrak{Q}\in[0,\infty)^N.
\end{equation*}

\begin{figure}[!ht]
	\begin{center}\vspace{-2mm}
		\input{fig1}\vspace{-6mm}
		\caption{Illustration of the single-file motion system with periodic boundary conditions. Here, $q_n$ represents the curvilinear position, $Q_n = q_{n+1} - q_n$ is the distance to the right neighbour and $p_n$ denotes the velocity of the $n$-th vehicle.}
		\label{fig:scheme}
	\end{center}
\end{figure}
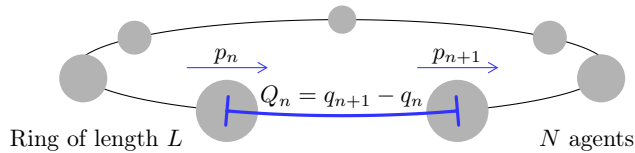

\subsection{Initial stochastic motion model}
First, we introduce a probability space $(\Omega, \mathcal{F}, \mathbb{P})$ 
on which there exists an $N$-dimensional standard Brownian motion 
\begin{equation*}
W=\bigl(W_n\bigr)_{n=1}^N\colon [0,\infty)\times \Omega \to \R^N.
\end{equation*}
For the $n$-th agent at time $t\in [0,\infty)$, the model reads
\begin{equation}\label{eq:modn}
\begin{aligned}
\begin{cases}
dQ_n(t)&=\bigl(p_{n+1}(t)-p_n(t)\bigr)\,dt,\\[1mm]
dp_n(t)&=\bigl(U'(Q_n(t))-U'(Q_{n-1}(t))\bigr)\,dt
+\;\beta\bigl(p_{n+1}(t)-2p_n(t)+p_{n-1}(t)\bigr)\,dt+\sigma \,dW_n(t),
\end{cases}\\ 
Q(0)=\mathfrak{Q},\quad p(0)=\mathfrak{p},
\end{aligned}
\end{equation}
where $\beta\in (0,\infty)$ is the dissipation rate, $\sigma\in\R$ is the noise volatility, and $U'$ is the derivative of a convex potential $U\in C^1(\R,[0,\infty))$. 
In the following, 
we use the quadratic functional 
\begin{equation}\label{eq:QuadraPot}
U(x)=\frac{1}{2}(\alpha x)^2, \quad x\in \R, \qquad\alpha \in (0,\infty). 
\end{equation}

\begin{remark}\label{rem:average_velocity}
	The ensemble's mean velocity
	\begin{equation}\label{eq:average_velocity}
	\xbar{p}(t)
	=\frac{1}{N}\sum_{n=1}^N p_n(t), \quad t\in [0,\infty),
	\end{equation}
	is a Brownian motion with variance $\sigma^2/N$ for any potential function $U$. 
	In fact, thanks to the telescopic form of the model~\eqref{eq:modn} and the periodic boundaries, we have
	\begin{equation*}\begin{aligned}
	d\xbar{p}(t)=\frac{1}{N}\sum_{n=1}^N dp_n(t)
	=\frac{\sigma}N\sum_{n=1}^N dW_n(t), \quad t\in [0,\infty).
	\end{aligned}
	\end{equation*}
	In addition, the ensemble's mean velocity \eqref{eq:average_velocity} is conserved for the deterministic system with $\sigma=0$.
\end{remark}

\subsection{Extended stochastic motion model}
The motion model \eqref{eq:modn} is symmetric and there is no inherent preference or favoured direction of motion.
In the following model we introduce a \textit{relaxation process} to a desired velocity $u\in \R$ with relaxation rate $\gamma \in [0,\infty)$. 
The extended motion model reads for the $n$-th agent at time $ t\in [0,\infty)$
\begin{equation}\label{eq:modn2}
\begin{aligned}
\begin{cases}
dQ_n(t)&=\bigl(p_{n+1}(t)-p_n(t)\bigr)\,dt,\\[1mm]
dp_n(t)&=\bigl(U'(Q_n(t))-U'(Q_{n-1}(t))\bigr)\,dt
+\beta\bigl(p_{n+1}(t)-2p_n(t)+p_{n-1}(t)\bigr)\,dt\\
&\quad +\, \gamma\bigl(u-p_n(t)\bigr)\,dt+\sigma \,dW_n(t),
\end{cases}\\ 
Q(0)=\mathfrak{Q},\quad p(0)=\mathfrak{p} .
\end{aligned}
\end{equation}
Note that the extended model \eqref{eq:modn2} recovers the model \eqref{eq:modn} for $\gamma=0$.

\begin{remark}\label{rem:average_velocity2}
	The ensemble's mean velocity \eqref{eq:average_velocity}
	follows an Ornstein-Uhlenbeck process for any potential function $U$. 
	Thanks to the telescopic form of the model~\eqref{eq:modn2} and the
	periodic boundaries we have
	\begin{equation*}\begin{aligned}
	d\xbar{p}(t) 
	=\gamma\bigl(u-\xbar{p}(t)\bigr)\,dt+\frac{\sigma}N\sum_{n=1}^N dW_n(t), \quad t\in [0,\infty).
	\end{aligned}
	\end{equation*}
	In addition, the ensemble's mean velocity \eqref{eq:average_velocity} relaxes to $u$ for the deterministic system with $\sigma=0$.
\end{remark}

\subsection{Port-Hamiltonian formulation}
Next, we rewrite the extended system \eqref{eq:modn2} in matrix form and identify a port-Hamiltonian structure.
\begin{proposition}\label{prop:PHSform}
	Let $Z(t)=(Q(t),p(t))^\top\in\R^{2N}$, $t\in[0,\infty)$.
	Then the dynamics of the periodic system \eqref{eq:modn2} are given by
	\begin{equation}\label{eq:PHS}
	\begin{aligned}
	dZ(t)=(J-R)\nabla H(Z(t))\,dt+Su\,dt+G\, dW(t),\\[0mm] 
	Z(0)=z(0)=(\mathfrak{Q},\mathfrak{p})^\top\in\R^{2N},
	\end{aligned}
	\end{equation}
	with 
	\begin{equation*}
	J=\begin{bmatrix}0&A\\-A^\top &0 \end{bmatrix}\in \R^{2N\times 2N},
	\quad 
	R=\begin{bmatrix}0&0\\0 &\;\beta A^\top A +\gamma I\end{bmatrix}\in \R^{2N\times 2N},
	\end{equation*}
	\begin{equation}\label{eq:def_A}
	A=\begin{bmatrix}-1&1& &\\
	&\ddots&\ddots&\\[1mm]
	&&-1&1\\
	1&&&-1\end{bmatrix}\in \R^{N\times N},
	\end{equation}
	\begin{equation*}
	S=\left(\begin{matrix}0\\\gamma \mathbf{1}\end{matrix}\right)\in \R^{2N},\quad
	G=\begin{bmatrix}0\\\sigma I\end{bmatrix}\in \R^{2N\times N},
	\end{equation*}
	$\mathbf{1}=(1,\ldots,1)\in\R^N$ and the Hamiltonian operator $H\colon \R^{2N}\to\R$,
	\begin{equation}\label{eq:def_Hamiltonian}
	H(Q,p)=\frac{1}{2} \|p\|^2+\sum_{n=1}^N U(Q_n), \quad p\in \R^N,\;Q \in \R^N.
	\end{equation}
	The matrix $J$ is skew-symmetric by $N\times N$ block,
	while $R$ is symmetric positive semi-definite.
\end{proposition}
\begin{proof}
	First note that for all $(Q,p)^\top \in \R^{2N}$ we have
	\begin{equation*}
	(\nabla H)(Q,p)=\begin{bmatrix}U'(Q)\\p\end{bmatrix},\quad\text{with}~~U'(Q)=\bigl(U'(Q_n)\bigr)_{n=1}^N.
	\end{equation*}
	We also have
	\begin{equation*}
	A^\top A=\begin{bmatrix}2&-1&&&-1\\
	-1&2&-1&&\\
	&\ddots&\ddots&\ddots&\\[1mm]
	&&-1&2&-1\\
	-1&&&-1&2\end{bmatrix}\in\R^{N\times N}.
	\end{equation*}
	It follows directly from the extended model \eqref{eq:modn2} that
	\begin{equation}\label{eq:dyn_Q_p}
	\begin{cases}
	dQ(t)&=A p(t)\,dt,\\
	dp(t)&=\bigl(-A^\top U'(Q(t))-(\beta A^\top A+\gamma I) p(t)\bigr)\,dt
	+\;\gamma u \,dt+ \sigma\,dW(t),
	\end{cases}
	\end{equation}
	and therefore
	\begin{equation*}
	\begin{aligned}
	dZ(t)=\tilde{B}\nabla H(Z(t))\,dt + Su\,dt + G\,dW(t)
	\end{aligned}
	\end{equation*}
	with
	\begin{equation*}
	\tilde{B}:=\begin{bmatrix}0&A\\-A^\top &-\beta A^\top A -\gamma I\end{bmatrix}=J-R\quad\text{and}\quad
	S=\begin{bmatrix}0\\\gamma I\end{bmatrix}
	.
	\end{equation*}
	It is clear that $J$ is skew-symmetric and $R$ is symmetric. 
	Furthermore, it holds for all $(Q,p)^\top \in \R^{2N}$ that 
	\begin{equation*}
	(Q^\top, p^{\top})R\begin{pmatrix}
	Q\\ p
	\end{pmatrix}
	=\beta \|A p\|^2+\gamma\|p\|^2\in [0,\infty)
	\end{equation*}
	and therefore $R$ is positive semi-definite. 
\end{proof}

\begin{remark}
	In the model, the Hamiltonian and the skew-symmetric matrix $J$ represent the conservative part of the system. 
	The velocity terms and the relaxation term to the desired velocity are part of the dissipation matrix $R$, the input matrix $S$, and the input control $u$. 
	Note that the model \eqref{eq:modn} with $\gamma=0$ is a stochastic Hamiltonian dissipative system while the extended model \eqref{eq:modn2} with $\gamma>0$ is a stochastic port-Hamiltonian system with input $u$.
\end{remark}

\begin{remark}
	Using the skew-symmetry of $J$ and the It\^o formula, we obtain for the Hamiltonian $H$ given by \eqref{eq:def_Hamiltonian}, after simplifications, that
	\begin{equation}\label{eq:dynHam}
	\begin{aligned}
	dH(Z(t))=&\bigl( -\beta \|A p(t)\|^2 +\gamma \bigl\langle p(t),u \mathbf{1}-p(t) \bigr\rangle\bigr)\,dt
	+\frac{N\sigma^2 }{2}dt + \sigma p^{\top}(t)\, dW(t).
	\end{aligned}
	\end{equation}
	Note that thanks to the skew symmetry, the Hamiltonian behaviour does not depend on the positions of the agents 
	and the distance-dependent interaction potential $U$ (parameter $\alpha$).
\end{remark}

\section{Long-time behaviour}\label{long}
We now consider the dynamics of the difference to the desired speed $u$
\begin{equation}
\tilde{p}_n(t)=p_n(t)-u,\quad t\in[0,\infty),\; n\in\{1,\ldots,N\},
\end{equation}
$\tilde{p}(t)=\bigl(\tilde{p}_n(t)\bigr)_{n=1}^N$ 
and $\tilde{Z}(t)=(Q(t),\tilde{p}(t))^\top$. 
Note that $\tilde{Z}(t)=Z(t)$ for the initial model \eqref{eq:modn} for which $u=0$. 

The process $\tilde{Z}$ satisfies
\begin{equation}\label{eq:def_tilde_z}
d\tilde{Z}(t)= B \tilde{Z}(t)\,dt+G\,dW(t), 
\quad t\in [0,\infty), 
\end{equation}
with
\begin{equation*}
B=\begin{bmatrix}0&A\\-\alpha^2 A^\top &-\beta A^\top A -\gamma I\end{bmatrix} \text{ and } 
G=\begin{bmatrix}0\\\sigma I\end{bmatrix}.
\end{equation*}
Furthermore, the Hamiltonian time derivative for the deterministic system with $\sigma=0$ satisfies
\begin{equation}\label{eq:Hdecreasing}
\frac{dH(\tilde{Z}(t))}{dt} =\ -\beta \|A \tilde{p}(t)\|^2 -\gamma \| 
\tilde{p}(t)\|^2 \le 0.
\end{equation}

\begin{remark}\label{rem:StabDet} 
	The uniform equilibrium solution for which $Q_n=L/N$ and $p_n=\frac1N\sum_{j=1}^Np_j(0)$ for all $n\in\{1,\ldots,N\}$ is stable for the deterministic model~\eqref{eq:modn} with $\sigma=0$ and a quadratic potential \eqref{eq:QuadraPot} for all $\alpha \in (0,\infty)$ and $\beta \in (0,\infty)$.
\end{remark}

\begin{remark}\label{rem:StabDet2} 
	The uniform equilibrium solution for which $Q_n=L/N$ and $p_n=u$ for all $n\in\{1,\ldots,N\}$ is stable for the deterministic model~\eqref{eq:modn2} with $\sigma=0$ and a quadratic potential \eqref{eq:QuadraPot} for all $\alpha \in (0,\infty)$, $\beta \in (0,\infty)$ and $\gamma \in (0,\infty)$. 
\end{remark}

Indeed, the Hamiltonian $H$ given by \eqref{eq:def_Hamiltonian} is a convex operator which is minimal for the uniform configuration $Q_n^\text u=L/N$ and $p_n^\text u=p^\text u$ for all $n\in\{1,\ldots,N\}$, $p^\text u=\xbar p(0)$ for the initial model \eqref{eq:modn} and $p^\text u=u$ for the extended model \eqref{eq:modn2}. 
Then the decrease of the Hamiltonian with time \eqref{eq:Hdecreasing} provides the stability of the deterministic systems.

Interestingly, both the deterministic motion models \eqref{eq:modn} and \eqref{eq:modn2} with a quadratic potential \eqref{eq:QuadraPot} and $\sigma=0$ have stable uniform solutions. 
However, only the extended model~\eqref{eq:modn2} with a control input $u$ in a port-Hamiltonian framework remains stable when stochastically perturbed. 
In the next result we show a weak convergence to a Gaussian distribution. 
In Proposition~\ref{prop:cov_matrix} we determine the covariance matrix of this stationary distribution and in Corollary~\ref{cor:cov_matrix_N_infty} we present its limit as the number of agents $N$ tends to infinity.

\begin{remark}
	In Proposition \ref{prop:ex_invariant_distr} below the spectral analysis of $B$ is extended based on the results given in
	Ehrhardt et al.~\cite{ehrhardt2023collective}.
	However, we can note beforehand that each block of $B$ is circulant and we have for $\lambda\in\C$
	\begin{equation*}
	|B-\lambda I_{2N}|=|C|,\quad C=\lambda^2 I_{N} + \lambda (\beta A^\top A+\gamma I_N) + \alpha^2 A^\top A.
	\end{equation*}
	The matrix $C=\begin{bmatrix}
	c_0 & c_{N-1} & \ldots & c_1\\
	c_1 & c_0 & \ldots & c_2\\
	\hdots & \hdots &  &\hdots \\
	c_{N-1} & c_{N-2} & \ldots & c_0
	\end{bmatrix}$, with $c_0=\lambda^2 + \lambda (2\beta + \gamma) + 2\alpha^2$, $c_1=c_{N-1}=-\lambda\beta-\alpha^2$, and $c_j=0$ for all $j\in\{2,\ldots,N-2\}$, is circulant. 
	Since the determinant of a circulant matrix is given by $|C|=\prod_{j=0}^{N-1} \bigl(c_0 + c_{N-1} \omega^j + c_{N-2} \omega^{2j} + \dots + c_1\omega^{(N-1)j}\bigr)$, where $\omega=e^{2\pi\iu/N}$, we obtain the characteristic equation of the matrix $B$ as
	\begin{equation*}
	\prod_{j=0}^{N-1} \bigl( \lambda^2 + \lambda (2\beta + \gamma) + 2\alpha^2 - (\lambda\beta+\alpha^2)(\omega^j+\omega^{(N-1)j})\bigr)
	=\prod_{j=0}^{N-1} \bigl( \lambda^2 + \lambda (\beta\mu_j + \gamma) + \alpha^2\mu_j\bigr)=0,
	\end{equation*}
	where $\mu_j=2-\omega^j-\omega^{(N-1)j}=2-2\cos\left(\frac{2\pi j}{N} \right)$. From this we can directly compute the eigenvalues of $B$ leading to \eqref{eq:ev_B} in the proof of Proposition \ref{prop:ex_invariant_distr}.
\end{remark}

\begin{proposition}\label{prop:ex_invariant_distr}
	If $\gamma>0$
	then the process $\tilde{Z}$ converges weakly as $t\to\infty$ to a Gaussian distribution with mean vector 
	$\begin{bmatrix}
	L/N \mathbf{1} & 0
	\end{bmatrix}^\top \in \R^{2N}$.
\end{proposition}
\begin{proof}
	We first generalise some results of Ehrhardt et al.~\cite{ehrhardt2023collective} to the case $\gamma\neq 0$. In particular, we present the eigenvalue decomposition of the system matrix $B$. 
	To this end, let $\lambda_{0,1}=0$, $\lambda_{0,2}=-\gamma$ and for all $j\in \{1,\ldots, N-1\}$, $k\in \{1,2\}$ 
	let $\mu_j=2-2\cos\left(\frac{2\pi j}{N} \right)\in [0,4]$ and
	\begin{equation}\label{eq:ev_B}
	\lambda_{j,k}=\frac{1}{2}\bigr(-(\beta\mu_j+\gamma)+(-1)^k\sqrt{(\beta\mu_j+\gamma)^2-4\alpha^2 \mu_j}\bigr)\in \C.
	\end{equation}

	We have $B=\mathcal W \Lambda \mathcal W^{-1}$, where $\mathcal W\in \C^{2N\times 2N}$ is given in \cite[Setting 3.1]{ehrhardt2023collective} 
	and $\Lambda\in \C^{2N\times 2N}$ is the diagonal matrix with the vector of eigenvalues
	\begin{equation*}
	\begin{bmatrix}
	\lambda_{0,1} & \lambda_{1,1} & \ldots & \lambda_{N-1,1} & 
	\lambda_{0,2} & \lambda_{1,2} & \ldots & \lambda_{N-1,2} 
	\end{bmatrix} \in \C^{2N}
	\end{equation*}
	on its diagonal. 
	To see this note that $B$ only differs from the system matrix $B$ in \cite{ehrhardt2023collective} by the term $-\gamma I$ in the lower right block. 
	This entails that the statement of \cite[Lemma 3.3]{ehrhardt2023collective} also holds true for $B$ with the only modification that $\lambda_1,\lambda_2\in \C$ are now the complex roots of $z\mapsto z^2+(\beta\kappa \tilde{\kappa}+\gamma)z+\alpha^2\kappa \tilde{\kappa}$. 
	From this it follows using the same arguments as in \cite[Lemma 3.5]{ehrhardt2023collective} that $B=\mathcal W \Lambda \mathcal W^{-1}$. 
	Note that $\lambda_{0,1}=0$ and that all other eigenvalues $\lambda_{j,k}$ have strictly negative real part (in particular, in contrast to the situation in \cite{ehrhardt2023collective}, we have $\lambda_{0,2}=-\gamma<0$). 
	
	Next define the process $Y(t)=\mathcal{W}^{-1}\tilde{Z}(t)$, $t\in [0,\infty)$. 
	Then $Y$ solves the complex-valued stochastic differential equation $dY(t) = \Lambda Y(t)\,dt+\mathcal{W}^{-1}G\,dW(t)$. 
	Note that \cite[Lemma 3.5]{ehrhardt2023collective} provides a closed-form representation of $\mathcal{W}^{-1}$. 
	In particular it follows that the first row of $\mathcal{W}^{-1}G$ is zero (this is due to the fact that the first row of $\mathcal{W}^{-1}$ is given by $u^*_{0,1}=\begin{pmatrix}
	v^*_0 & 0
	\end{pmatrix}$, see \cite[Setting 3.1]{ehrhardt2023collective}). 
	This together with the fact that $\lambda_{1,0}=0$ implies that the first component $Y_1$ of $Y$ is constant equal to $Y_1(0)$. 
	The remaining diagonal entries of $\Lambda$ have strictly negative real parts. 
	From this we conclude that the process $(\Re(Y(t)), \Im(Y(t)))$, $t\in [0,\infty)$, is (degenerate) Gaussian and converges weakly as $t\to\infty$ to its unique invariant distribution. 
	It follows that also $\tilde{Z}(t) = \Re(\tilde{Z}(t))=\Re(\mathcal{W}) \Re(Y(t)) - \Im(\mathcal{W}) \Im(Y(t))$, $t\in [0,\infty)$, converges weakly as $t\to\infty$ to its unique invariant distribution. 
	Since the invariant distribution of $Y$ has a mean vector 
	$(Y_1(0),0)^\top$, 
	we obtain that the mean vector of the invariant distribution of 
	$\tilde{Z}$ is given by 
	\begin{equation*}
	\begin{split}
	&    \mathcal{W} \begin{pmatrix}
	Y_1(0) \\ 0
	\end{pmatrix} = Y_1(0)w_{0,1}=(u_{0,1}^*\tilde{Z}(0))w_{0,1}
	=\frac{1}{\sqrt{N}}\sum_{k=1}^N\tilde{Z}_k(0) \frac{1}{\sqrt{N}} \binom{\mathbf{1}}{0}
	=\frac{L}{N}\binom{\mathbf{1}}{0}
	\end{split}
	\end{equation*}
	where we used the notation of \cite[Setting 3.1]{ehrhardt2023collective} and the fact that $\sum_{k=1}^N\tilde{Z}_k(0)=L$. 
	This concludes the proof.
\end{proof}

The covariance matrix $\Sigma$ of the limiting Gaussian distribution of Proposition~\ref{prop:ex_invariant_distr} necessarily satisfies the matrix Lyapunov equation associated with the system $\tilde{Z}$ (see, e.g., Gardiner~\cite[Section 4.4.6]{gardiner1985handbook} or Pavliotis~\cite[Section 3.7]{pavliotis2014stochastic}). 
In our setting this equation is (see \eqref{eq:def_tilde_z})
\begin{equation}\label{eq:Lyap1}
\B\Sigma+\Sigma\B^\top=-G G^\top.
\end{equation}
In the next steps we will show that \eqref{eq:Lyap1} uniquely determines $\Sigma$ and derive its closed-form representation.
To do this we write  
\begin{equation}\label{eq:ansatz_cov}
\Sigma=\begin{bmatrix}
V_1 & V_2\\ V_2^\top & V_3
\end{bmatrix}\in \R^{2N\times 2N}   
\end{equation}
for some $V_2\in \R^{N\times N}$ and symmetric $V_1,V_3\in \R^{N\times N}$. 
Then \eqref{eq:Lyap1} is equivalent to 
\begin{align}\label{eq:Lyap2}
&AV_2^\top=-V_2A^\top, \,
AV_3-\alpha^2 V_1 A - \beta V_2 A^\top A =\gamma V_2 \\
\label{eq:Lyap3}
&\alpha^2(A^\top V_2+V_2^\top A)+\beta (A^\top AV_3 + V_3 A^\top A) +2\gamma V_3 =\sigma^2 I .
\end{align}
From the symmetry of the system we conclude that the matrices $V_1,V_2,V_3$ are circulant and that $V_2$ is also symmetric. 
In particular, we have that $V_3$ is determined by a vector $v=(v_0,\ldots,v_{N-1})^\top\in \R^N$ via
\begin{equation}\label{eq:V3}
V_3=\begin{bmatrix}
v_0 & v_{N-1} & \ldots & v_1\\
v_1 & v_0 & \ldots & v_2\\
\hdots & \hdots & \hdots &\hdots \\
v_{N-1} & v_{N-2} & \ldots & v_0
\end{bmatrix} 
\end{equation}
and similarly $V_1$ and $V_2$. 
Since circulant matrices commute, we see that \eqref{eq:Lyap3} is equivalent to $\beta A^\top AV_3+\gamma V_3=\frac{\sigma^2}{2} I$. 
Since $V_3$ is completely determined by $v$, it suffices to look at the first column of this matrix equation which is $(\beta A^\top A+\gamma I)v=\frac{\sigma^2}{2} e_1$, where $e_1\in \R^N$ is the first standard unit vector. 
Since $(\beta A^\top A+\gamma I)$ is circulant, this linear system is equivalent to the convolution equation $w * v=\frac{\sigma^2}{2} e_1$, where $w=\begin{pmatrix}
2\beta+\gamma & -\beta & 0 & \ldots 0 & -\beta   
\end{pmatrix}^\top\in \R^N$ is the first column of $\beta A^\top A+\gamma I$. 
Next we apply the discrete Fourier transform $\mathcal{F}_N$ to this equation to obtain for each $k\in \{0,1,\ldots, N-1\}$ that 
$[\mathcal{F}_N(w)](k)[\mathcal{F}_N(v)](k)=\frac{\sigma^2}{2}[\mathcal{F}_N(e_1)](k)$. 
Using that $[\mathcal{F}_N(w)](k)=2\beta+\gamma-\beta(e^{-2\pi \iu k/N}+e^{2\pi \iu k/N})$
and $[\mathcal{F}_N(e_1)](k)=1$ and applying the inverse Fourier transform gives
\begin{equation}\label{eq:v_j}
\begin{split}
v_j&=\frac{\sigma^2}{2N}\sum_{k=0}^{N-1}\frac{e^{2\pi \iu jk/N}}{2\beta+\gamma-\beta(e^{-2\pi \iu k/N}+e^{2\pi \iu k/N})}
=\frac{\sigma^2}{2N}\sum_{k=0}^{N-1}\frac{\cos(2\pi jk/N)}{\gamma+4\beta\sin^2(\pi k/N)}.
\end{split}
\end{equation}
This determines $V_3$. 

Next we use \eqref{eq:Lyap2} to show that $V_2=0$. Note that the fact that $V_2$ is symmetric and circulant, together with the first equation in \eqref{eq:Lyap2}, ensures that it satisfies $(A+A^\top)V_2=0$. 
Since the kernel of $A+A^\top$ is spanned by $\mathbf{1}$, we conclude that $V_2$ is a multiple of $\overline{\mathbf{1}}$, where $\overline{\mathbf{1}}\in \R^{N\times N}$ is the matrix consisting only of $1$s. 
Multiplying the second equation in \eqref{eq:Lyap2} by $\mathbf{1}$ (remember that $\mathbf{1}=(1,\ldots,1)^\top \in \R^N$) implies $\gamma V_2 \mathbf{1}=0$, which finally gives $V_2=0$.

The block $V_1$ can now be derived from the second equation in \eqref{eq:Lyap2}. 
Again, due to the cyclicity of $V_3$, this equation is equivalent to $V_3A=\alpha^2V_1A$, and so we can read off the solution $V_1=\frac{1}{\alpha^2}V_3$. Note, however, that $A$ is not regular, but only of rank $N-1$, with $\mathbf{1}\in \R^N$ being the basis of its kernel. 
From this we conclude that the solution set of $V_1=\frac{1}{\alpha^2}V_3$ is given by $\{\alpha^{-2}(V_3+\kappa \overline{\mathbf{1}})|\kappa \in \R\}$.
Next, note that for all $t\in [0,\infty)$ we have that $\sum_{n=1}^NQ_n(t)=L$. 
This implies that $0=\mathbf{1}^\top V_1 \mathbf{1}=\alpha^{-2}\mathbf{1}^\top (V_3+\kappa \overline{\mathbf{1}})\mathbf{1}=\alpha^{-2}(N\sum_{j=0}^{N-1}v_j+\kappa N^2)$. 
This implies that 
$\kappa=-\frac{1}{N}\sum_{j=0}^{N-1}v_j=-\frac{\sigma^2}{2\gamma N}$. So we have $V_1=\alpha^{-2}(V_3-\frac{\sigma^2}{2\gamma N} \mathbf{1})$. 
We summarise these results in the following proposition.
\begin{proposition}\label{prop:cov_matrix}
	Suppose $\gamma>0$. Then
	the covariance matrix of the limiting distribution of $\tilde{Z}$ (cf.\ Proposition~\ref{prop:ex_invariant_distr}) is given by \eqref{eq:ansatz_cov} with $V_3$ given by \eqref{eq:V3} and $v$ by \eqref{eq:v_j}, $V_2=0$ and $V_1=\alpha^{-2}(V_3-\frac{\sigma^2}{2\gamma N} \mathbf{1})$.
\end{proposition}

Next, we present the limit of the covariance matrix as the number of agents $N$ tends to infinity.

\begin{corollary}\label{cor:cov_matrix_N_infty}
	Suppose $\gamma>0$ and for all $N\in\{3,4,\ldots\}$, $j\in {\{0,1,\ldots,N-1\}}$ let $v_j^N$ be given by~\eqref{eq:v_j}.
	Then it holds for all $\tilde{N} \in \{3,4,\ldots\}$, $j\in\{0,1,\ldots,\tilde{N}-1\}$ that $(v_j^N)_{N\in\{\tilde{N},\tilde{N}+1,\ldots\}}$ converges to $\frac{\sigma^2 a^j}{2F}$ with 
	$a=1+\frac{\gamma}{2\beta}-\frac{F}{2\beta}$ and $F=\sqrt{\gamma^2+4\beta\gamma}$.
\end{corollary}

\begin{proof}
	Let 
	$a=1+\frac{\gamma}{2\beta}-\frac{F}{2\beta}$ and $b=1+\frac{\gamma}{2\beta}+\frac{F}{2\beta}$ with $F=\sqrt{\gamma^2 + 4 \beta \gamma}$, and observe 
	that $a\in(0,1)$ and $b>1$. 
	For all $j\in\N_0$ let 
	$g_j\colon \{z\in\C\colon \lvert z \rvert < b\} \to \C$, $g_j(z)=\frac{\sigma^2 z^j}{2\beta(b-z)}$ 
	and note that $g_j$ is a holomorphic function on $\{z\in\C\colon \lvert z \rvert < b\}$. 
	Further,  
	let $\Gamma\colon [0,1]\to \C$, $\Gamma(t)=e^{2\pi \iu t}$. 
	Observe for all $N\in\{3,4,\ldots\}$, $j \in \{0,1,\ldots,N-1\}$ 
	that 
	$v_j \equiv v_j^N$ in~\eqref{eq:v_j} 
	is a Riemann sum and 
	that it holds that 
	\begin{equation*}
	v_j^N = \frac{1}{N} \sum_{k=1}^N 
	\frac{ g_j\Bigl(\Gamma\bigl(\frac{k}{N}\bigr) \Bigr) \Gamma\bigl(\frac{k}{N}\bigr)}{\Gamma\bigl(\frac{k}{N}\bigr) - a}  .
	\end{equation*}
	We thus consider for all $j\in\N_0$ the line integral
	\begin{equation*}
	\begin{split}
	\frac{1}{2\pi \iu} \int_{\Gamma} \frac{g_j(z)}{z-a}\,dz 
	& = \int_0^1 \frac{g_j\bigl(\Gamma(t)\bigr) \Gamma(t)}{\Gamma(t) - a}\,dt.
	\end{split}
	\end{equation*}
	By Cauchy's integral formula, we have for all $j\in\N_0$ that 
	\begin{equation*}
	\frac{1}{2\pi \iu} \int_{\Gamma} \frac{g_j(z)}{z-a}\,dz  
	= g_j(a) = \frac{\sigma^2a^j}{2F} .
	\end{equation*}
	Therefore, we obtain for all $j \in \N_0$ that 
	\begin{equation*}
	\begin{split}
	& \lim_{N\to\infty} \frac{1}{N} \sum_{k=1}^N \frac{ g_j\Bigl( \Gamma\bigl(\frac{k}{N}\bigr) \Bigr) \Gamma\bigl(\frac{k}{N}\bigr)}{\Gamma\bigl(\frac{k}{N}\bigr) - a} 
	= \frac{\sigma^2a^j}{2F} 
	= \frac{\sigma^2 \biggl( 1+\frac12 \frac{\gamma}{\beta} - \frac12  \sqrt{\bigl( \frac{\gamma}{\beta} \bigr)^2 + 4 \frac{\gamma}{\beta}}  \biggr)^j}{2\beta \sqrt{\bigl( \frac{\gamma}{\beta} \bigr)^2 + 4 \frac{\gamma}{\beta}}} .
	\end{split}
	\end{equation*}
\end{proof}

\section{Simulation results}\label{sim}
In this section we present simulations of $N=10$ agents moving according to the dynamics \eqref{eq:modn} and \eqref{eq:modn2} along a segment of length $L=501$ with periodic boundaries and $\alpha=\beta=\sigma=1$. 
The simulations are performed from a uniform zero velocity initial condition using an Euler-Maruyama scheme with a time step of $dt=0.001$. 
The trajectories over the first 500 time units are shown in Fig.~\ref{fig:traj}. 
The upper panel shows the trajectories for the diverging dynamics \eqref{eq:modn} with $\gamma=0$, where the ensemble's mean velocity follows a Brownian motion.
The middle and lower panels show the trajectories of the extended port-Hamiltonian model with $\gamma=0.1$ and $\gamma=1$, respectively, where the velocities of the vehicles remain close to the control input velocity $u=0$.

The simulations can be run in real time on the online platform at: 
 \url{https://www.vzu.uni-wuppertal.de/fileadmin/site/vzu/Simulating_Collective_Motion.html?speed=0.7}.

\begin{figure}[!ht]
	\centering\vspace{-3mm}
	\input{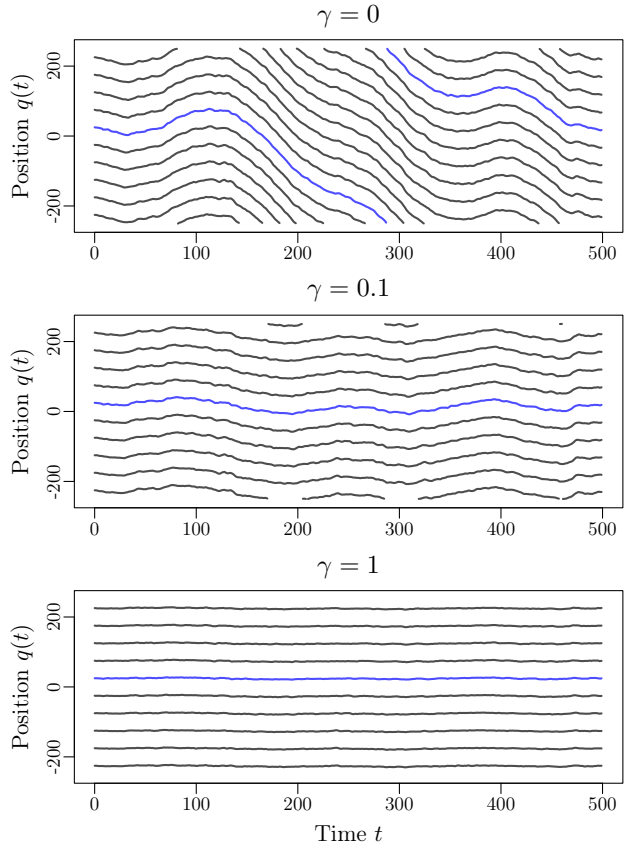}\vspace{-0mm}
	\caption{Trajectories of 10 vehicles along a segment of length $L=501$ with periodic boundaries with the dynamics \eqref{eq:modn2} with $\gamma=0$ (the ensemble's mean velocity diverges, upper panel) and with the pHS with $\gamma=0.1$ and $\gamma=1$ and the controlled input velocity $u=0$ (middle and lower panels).}
	\label{fig:traj}
	\medskip
\end{figure}

	
\section{Conclusion}\label{ccl}
This study has revisited and analysed a symmetric port-Hamiltonian single-file model in one dimension, emphasising the impact of white noise on its stability. 
The introduction of a control term within the port-Hamiltonian framework has proven effective in stabilising the dynamics, even in the presence of noise. 
These findings offer valuable insights for enhancing the control of road traffic flows. 

Future research may explore time-dependent control strategies ($u=u(t)$), distance-dependent control ($u=u(Q_n)$) as seen in the model by 
Rüdiger et al.~\cite{rudiger2022stability}, agent-specific controls ($u=u_n$), and random controls with uncertainties ($u=U$). These potential extensions provide a rich landscape for further investigation and refinement, contributing to the ongoing development of robust and adaptable models for traffic flow control.

	\bibliographystyle{abbrv}
	\bibliography{refs}
	
\end{document}

%% file: fig1.tex
\small\hspace{-5mm}
\begin{tikzpicture}[x=1.07pt,y=1.07pt]
\definecolor{fillColor}{RGB}{255,255,255}
\path[use as bounding box,fill=fillColor,fill opacity=0.00] (0,0) rectangle (247.08, 77.21);
\begin{scope}
\path[clip] ( 5.50, 5.50) rectangle (235.08, 65.21);
\definecolor{drawColor}{RGB}{0,0,0}

\path[draw=drawColor,line width= 0.4pt,line join=round,line cap=round] (123.54, 58.14) --
	(126.50, 58.14) --
	(129.46, 58.11) --
	(132.42, 58.07) --
	(135.36, 58.01) --
	(138.30, 57.93) --
	(141.22, 57.84) --
	(144.12, 57.73) --
	(147.00, 57.60) --
	(149.86, 57.46) --
	(152.69, 57.30) --
	(155.50, 57.13) --
	(158.27, 56.94) --
	(161.00, 56.73) --
	(163.70, 56.51) --
	(166.36, 56.27) --
	(168.98, 56.02) --
	(171.55, 55.75) --
	(174.07, 55.47) --
	(176.54, 55.18) --
	(178.96, 54.87) --
	(181.33, 54.54) --
	(183.64, 54.21) --
	(185.88, 53.86) --
	(188.07, 53.49) --
	(190.19, 53.12) --
	(192.24, 52.73) --
	(194.23, 52.34) --
	(196.14, 51.93) --
	(197.99, 51.51) --
	(199.76, 51.08) --
	(201.45, 50.63) --
	(203.06, 50.19) --
	(204.60, 49.73) --
	(206.05, 49.26) --
	(207.43, 48.78) --
	(208.72, 48.30) --
	(209.92, 47.81) --
	(211.04, 47.31) --
	(212.07, 46.81) --
	(213.01, 46.30) --
	(213.87, 45.79) --
	(214.63, 45.27) --
	(215.30, 44.75) --
	(215.88, 44.22) --
	(216.37, 43.69) --
	(216.77, 43.16) --
	(217.07, 42.63) --
	(217.28, 42.09) --
	(217.40, 41.56) --
	(217.42, 41.02) --
	(217.35, 40.48) --
	(217.19, 39.95) --
	(216.93, 39.41) --
	(216.58, 38.88) --
	(216.14, 38.35) --
	(215.60, 37.82) --
	(214.98, 37.30) --
	(214.26, 36.78) --
	(213.45, 36.26) --
	(212.55, 35.75) --
	(211.57, 35.25) --
	(210.49, 34.75) --
	(209.33, 34.25) --
	(208.08, 33.77) --
	(206.75, 33.29) --
	(205.34, 32.81) --
	(203.84, 32.35) --
	(202.27, 31.90) --
	(200.61, 31.45) --
	(198.88, 31.02) --
	(197.07, 30.59) --
	(195.20, 30.18) --
	(193.24, 29.77) --
	(191.22, 29.38) --
	(189.14, 29.00) --
	(186.98, 28.63) --
	(184.77, 28.27) --
	(182.49, 27.93) --
	(180.15, 27.60) --
	(177.76, 27.28) --
	(175.31, 26.98) --
	(172.82, 26.69) --
	(170.27, 26.42) --
	(167.67, 26.16) --
	(165.04, 25.91) --
	(162.36, 25.68) --
	(159.64, 25.47) --
	(156.88, 25.27) --
	(154.10, 25.09) --
	(151.28, 24.92) --
	(148.43, 24.77) --
	(145.56, 24.64) --
	(142.67, 24.52) --
	(139.76, 24.42) --
	(136.83, 24.34) --
	(133.89, 24.27) --
	(130.94, 24.22) --
	(127.98, 24.18) --
	(125.02, 24.17) --
	(122.06, 24.17) --
	(119.09, 24.18) --
	(116.14, 24.22) --
	(113.18, 24.27) --
	(110.24, 24.34) --
	(107.32, 24.42) --
	(104.40, 24.52) --
	(101.51, 24.64) --
	( 98.64, 24.77) --
	( 95.80, 24.92) --
	( 92.98, 25.09) --
	( 90.19, 25.27) --
	( 87.44, 25.47) --
	( 84.72, 25.68) --
	( 82.04, 25.91) --
	( 79.40, 26.16) --
	( 76.81, 26.42) --
	( 74.26, 26.69) --
	( 71.76, 26.98) --
	( 69.32, 27.28) --
	( 66.92, 27.60) --
	( 64.59, 27.93) --
	( 62.31, 28.27) --
	( 60.09, 28.63) --
	( 57.94, 29.00) --
	( 55.85, 29.38) --
	( 53.83, 29.77) --
	( 51.88, 30.18) --
	( 50.00, 30.59) --
	( 48.20, 31.02) --
	( 46.47, 31.45) --
	( 44.81, 31.90) --
	( 43.24, 32.35) --
	( 41.74, 32.81) --
	( 40.33, 33.29) --
	( 38.99, 33.77) --
	( 37.75, 34.25) --
	( 36.59, 34.75) --
	( 35.51, 35.25) --
	( 34.52, 35.75) --
	( 33.63, 36.26) --
	( 32.82, 36.78) --
	( 32.10, 37.30) --
	( 31.47, 37.82) --
	( 30.94, 38.35) --
	( 30.49, 38.88) --
	( 30.14, 39.41) --
	( 29.89, 39.95) --
	( 29.72, 40.48) --
	( 29.65, 41.02) --
	( 29.68, 41.56) --
	( 29.79, 42.09) --
	( 30.00, 42.63) --
	( 30.31, 43.16) --
	( 30.70, 43.69) --
	( 31.19, 44.22) --
	( 31.78, 44.75) --
	( 32.45, 45.27) --
	( 33.21, 45.79) --
	( 34.06, 46.30) --
	( 35.01, 46.81) --
	( 36.04, 47.31) --
	( 37.16, 47.81) --
	( 38.36, 48.30) --
	( 39.65, 48.78) --
	( 41.02, 49.26) --
	( 42.48, 49.73) --
	( 44.01, 50.19) --
	( 45.63, 50.63) --
	( 47.32, 51.08) --
	( 49.09, 51.51) --
	( 50.93, 51.93) --
	( 52.85, 52.34) --
	( 54.83, 52.73) --
	( 56.89, 53.12) --
	( 59.01, 53.49) --
	( 61.19, 53.86) --
	( 63.44, 54.21) --
	( 65.75, 54.54) --
	( 68.11, 54.87) --
	( 70.53, 55.18) --
	( 73.01, 55.47) --
	( 75.53, 55.75) --
	( 78.10, 56.02) --
	( 80.72, 56.27) --
	( 83.38, 56.51) --
	( 86.07, 56.73) --
	( 88.81, 56.94) --
	( 91.58, 57.13) --
	( 94.38, 57.30) --
	( 97.22, 57.46) --
	(100.07, 57.60) --
	(102.96, 57.73) --
	(105.86, 57.84) --
	(108.78, 57.93) --
	(111.71, 58.01) --
	(114.66, 58.07) --
	(117.61, 58.11) --
	(120.57, 58.14) --
	(123.54, 58.14);
\definecolor{fillColor}{gray}{0.70}

\path[fill=fillColor] ( 82.80, 25.85) circle ( 11);

\path[fill=fillColor] (164.27, 25.85) circle ( 11);

\path[fill=fillColor] (215.07, 37.37) circle ( 8.5);

\path[fill=fillColor] (196.94, 51.75) circle (  6);

\path[fill=fillColor] (123.54, 58.14) circle (  5);

\path[fill=fillColor] ( 50.13, 51.75) circle (  6);

\path[fill=fillColor] ( 32.00, 37.37) circle ( 8.5);
\definecolor{drawColor}{RGB}{50,50,255}

\path[draw=drawColor,line width= 1.2pt,line join=round,line cap=round] ( 82.80, 25.85) --
	( 83.57, 25.78) --
	( 84.34, 25.72) --
	( 85.12, 25.65) --
	( 85.90, 25.59) --
	( 86.68, 25.53) --
	( 87.46, 25.47) --
	( 88.25, 25.41) --
	( 89.04, 25.35) --
	( 89.83, 25.30) --
	( 90.63, 25.24) --
	( 91.43, 25.19) --
	( 92.23, 25.14) --
	( 93.03, 25.09) --
	( 93.84, 25.04) --
	( 94.65, 24.99) --
	( 95.46, 24.94) --
	( 96.27, 24.90) --
	( 97.09, 24.85) --
	( 97.91, 24.81) --
	( 98.73, 24.77) --
	( 99.55, 24.73) --
	(100.37, 24.69) --
	(101.20, 24.65) --
	(102.02, 24.62) --
	(102.85, 24.58) --
	(103.69, 24.55) --
	(104.52, 24.52) --
	(105.35, 24.49) --
	(106.19, 24.46) --
	(107.03, 24.43) --
	(107.86, 24.40) --
	(108.70, 24.38) --
	(109.55, 24.35) --
	(110.39, 24.33) --
	(111.23, 24.31) --
	(112.08, 24.29) --
	(112.92, 24.27) --
	(113.77, 24.26) --
	(114.61, 24.24) --
	(115.46, 24.23) --
	(116.31, 24.22) --
	(117.16, 24.20) --
	(118.01, 24.19) --
	(118.86, 24.19) --
	(119.71, 24.18) --
	(120.56, 24.17) --
	(121.41, 24.17) --
	(122.26, 24.17) --
	(123.11, 24.16) --
	(123.96, 24.16) --
	(124.82, 24.17) --
	(125.67, 24.17) --
	(126.52, 24.17) --
	(127.37, 24.18) --
	(128.22, 24.19) --
	(129.07, 24.19) --
	(129.92, 24.20) --
	(130.77, 24.22) --
	(131.62, 24.23) --
	(132.46, 24.24) --
	(133.31, 24.26) --
	(134.16, 24.27) --
	(135.00, 24.29) --
	(135.85, 24.31) --
	(136.69, 24.33) --
	(137.53, 24.35) --
	(138.37, 24.38) --
	(139.21, 24.40) --
	(140.05, 24.43) --
	(140.89, 24.46) --
	(141.72, 24.49) --
	(142.56, 24.52) --
	(143.39, 24.55) --
	(144.22, 24.58) --
	(145.05, 24.62) --
	(145.88, 24.65) --
	(146.71, 24.69) --
	(147.53, 24.73) --
	(148.35, 24.77) --
	(149.17, 24.81) --
	(149.99, 24.85) --
	(150.81, 24.90) --
	(151.62, 24.94) --
	(152.43, 24.99) --
	(153.24, 25.04) --
	(154.04, 25.09) --
	(154.85, 25.14) --
	(155.65, 25.19) --
	(156.45, 25.24) --
	(157.24, 25.30) --
	(158.04, 25.35) --
	(158.83, 25.41) --
	(159.62, 25.47) --
	(160.40, 25.53) --
	(161.18, 25.59) --
	(161.96, 25.65) --
	(162.73, 25.72) --
	(163.51, 25.78) --
	(164.27, 25.85);

\path[draw=drawColor,line width= 1.2pt,line join=round,line cap=round] ( 82.80, 25.85) -- ( 83.57, 25.78);

\path[draw=drawColor,line width= 1.2pt,line join=round,line cap=round] ( 83.24, 30.89) --
	( 82.80, 25.85) --
	( 82.37, 20.81);

\path[draw=drawColor,line width= 1.2pt,line join=round,line cap=round] (164.27, 25.85) -- (163.51, 25.78);

\path[draw=drawColor,line width= 1.2pt,line join=round,line cap=round] (164.71, 20.81) --
	(164.27, 25.85) --
	(163.84, 30.89);
\definecolor{drawColor}{RGB}{0,0,0}


\definecolor{drawColor}{RGB}{50,50,255}

\path[draw=drawColor,line width= 0.4pt,line join=round,line cap=round] ( 68.72, 41.15) -- ( 96.89, 41.15);

\path[draw=drawColor,line width= 0.4pt,line join=round,line cap=round] ( 93.76, 39.35) --
	( 96.89, 41.15) --
	( 93.76, 42.96);
\definecolor{drawColor}{RGB}{0,0,0}

\node[text=drawColor,anchor=base,inner sep=0pt, outer sep=0pt, scale=  0.90] at ( 82.80, 44.25) {$p_n$};

\definecolor{drawColor}{RGB}{50,50,255}

\path[draw=drawColor,line width= 0.4pt,line join=round,line cap=round] (150.19, 41.15) -- (178.36, 41.15);

\path[draw=drawColor,line width= 0.4pt,line join=round,line cap=round] (175.23, 39.35) --
	(178.36, 41.15) --
	(175.23, 42.96);
\definecolor{drawColor}{RGB}{0,0,0}

\node[text=drawColor,anchor=base,inner sep=0pt, outer sep=0pt, scale=  0.90] at (164.27, 44.25) {$p_{n+1}$};

\node[text=drawColor,anchor=base,inner sep=0pt, outer sep=0pt, scale=  0.90] at (123.54, 29) {$Q_n=q_{n+1}-q_n$};

\node[text=drawColor,anchor=base,inner sep=0pt, outer sep=0pt, scale=  0.90] at ( 36.69, 13.45) {Ring of length $L$};

\node[text=drawColor,anchor=base,inner sep=0pt, outer sep=0pt, scale=  0.90] at (210, 13.45) {$N$ agents};



\end{scope}
\end{tikzpicture}